\documentclass[11pt]{amsart}
\pdfoutput=1

\usepackage[utf8]{inputenc}

\usepackage[margin=3cm]{geometry}

\usepackage{amssymb, amsmath, amsthm, amsfonts}
\usepackage{graphicx}
\usepackage{color}
\usepackage{mathrsfs}
\usepackage{physics}
\usepackage{mathtools}
\usepackage{tikz}
\usepackage{faktor}
\usepackage{tikz-cd}
\usepackage{bm}
\usepackage{hyperref}
\usepackage{csquotes}

\DeclareRobustCommand{\SkipTocEntry}[5]{}

\numberwithin{equation}{section}

\newtheorem{theorem}{Theorem}[section]

\newtheorem{lemma}{Lemma}[section]
\newtheorem{corollary}{Corollary}[section]

\newtheorem{prop}{Proposition}[section]
\theoremstyle{remark}
\newtheorem{remark}{Remark}[section]
\theoremstyle{remark}
\newtheorem{example}{Example}[section]

\usepackage[UKenglish]{isodate}
\usepackage[UKenglish]{babel}

\newcommand{\R}{\mathbf{R}}

\DeclareMathOperator{\E}{\mathbf{E}}
\newcommand{\lag}{\mathcal{L}}
\newcommand{\G}{\mathcal{G}}
\DeclareMathOperator*{\argmax}{\text{arg\,max}}
\DeclareMathOperator*{\argmin}{\text{arg\,min}}

\newcommand{\F}{\mathscr{F}}

\newcommand{\eps}{\varepsilon}
\newcommand{\V}{\mathbf{V}}

\title{The relation of bias with risk in empirically constrained inferences}

\author{Dalton A R Sakthivadivel}
\address{Department of Mathematics, CUNY Graduate Centre, 365 Fifth Avenue, New York, NY 10016}
\email{dsakthivadivel@gc.cuny.edu}
\urladdr{https://darsakthi.github.io}

\date{\today}

\let\oldtocsection=\tocsection

\let\oldtocsubsection=\tocsubsection

\renewcommand{\tocsection}[2]{\hspace{0em}\oldtocsection{#1}{#2}}
\renewcommand{\tocsubsection}[2]{\hspace{19.5pt}\oldtocsubsection{#1}{#2}}

\begin{document}

\begin{abstract}

We give some results relating asymptotic characterisations of maximum entropy probability measures to characterisations of Bayes optimal classifiers. Our main theorems show that maximum entropy is a universally Bayes optimal decision rule given constraints on one's knowledge about some observed data in terms of an expected loss. We will extend this result to the case of uncertainty in the observations of expected losses by generalising Sanov's theorem to distributions of constraint values. 

\end{abstract}


\maketitle

\section{Introduction}

Given the probability space $(\Omega, \F, \mu)$, $\Omega$ a Polish space, let $X \sim P \coloneqq \mu_*$ be a random variable $X \colon \Omega \to M$, and $Y \colon \Pi \to N$ on $(\Pi, \mathscr{G}, \nu)$ be a {\it label} associated to $X$ by some unknown $\mu_*$-measurable 
mapping $f \colon X \mapsto Y$. One is often interested in approximating $f$ by some hypothetical parametric function $h_\theta \colon X \mapsto Y$, classifying observed values of $X$ according to their `best possible' labels. Given a loss function $\lag(h_\theta(x), y)$, the {\it risk of misclassification} or {\it generalisation error} is the expected error on the joint measure space
\[
    R(h_\theta) = \int_N\int_M \lag(h_\theta(x), y) p(x, y) \dd x \dd y.
\]
Under obvious lower semi-continuity assumptions on the risk there is a unique classifier $g$ of least risk such that for any $h_\theta$, 
\[
    R(h_\theta) \geqslant R(g).
\]
We call $g$ the Bayes classifier. It can be shown that the output of the Bayes classifier is determined by minimising the conditionally expected error,
\begin{equation}\label{expected-error}
    g(x) = \argmin_z \int_N \lag(z, y) p(y \mid x) \dd y.
\end{equation} 
In this case, one often says $g$ is Bayes optimal with respect to the prior probability of $y$, using the Bayesian inversion $p(x \mid y)\pi(y)$. For standard references see \cite{berger2013statistical, mackay2003information, shalev2014understanding}.

We claim one can turn this story around and see it as the inference of a posterior consistent with our knowledge of (i) the value of $X$ and (ii) the loss function on possible classes associated to $X$. In other words, from the existence of a Bayes classifier follows an optimal `guess' of the distribution over $Y$ given the information we have about $X$, such that the loss function and the distance of $h_\theta$ from $g$ constrains the form of $P(Y \mid X)$. The results here are motivated by this claim and the accompanying question: how should that inference be performed---what is the optimal such guess? Moreover, what determines optimality in this context?

The principle of maximum entropy---in its asymptotic form---has obvious relevance, envisioned by {\it e.g.} Skyrms as a method for supposition which is compatible with Bayesian model evidence evaluation \cite{skyrms1987updating} and by Shore--Johnson as an inference procedure consistent with some given knowledge but including no extraneous assumptions \cite{shore1980axiomatic}. We will presently describe how these properties are relevant to the matter at hand.

The concept behind \eqref{expected-error} is that for a given $X=x$ we want to predict a $z$ and compare it against any given $Y=y$ using $\lag$, then consider the total result across all $y$'s possibly associated to $x$, weighted by their likelihoods. Note that this lattermost property of \eqref{expected-error} can be thought of as controlling the risk of overfitting the classifier to a rare observation of $Y$ given $X$. The classifier which minimises this expected loss is the Bayes classifier. An example is that of a quadratic loss function $\lag(h_\theta(x), y) = (h_\theta(x) - y)^2$, in which case the best value for $h_\theta(x) = z$ can be computed as $\E[Y \mid X]$. By pure reasoning, one knows this is equal to the value of $Y$ observed on average given $X$, so minimises the conditionally expected difference between classifier outputs and values of $Y$, which is \eqref{expected-error}. The computation follows precisely this logic by setting the derivative of \eqref{expected-error} with respect to $z$, $2\E[z - y \mid X]$, equal to zero, and noting this is equivalent to $z - \E[Y \mid X] = 0$ by linearity and independence of the variable $z$ from $y$.

Suppose we make a random choice of $z$ weighted according to $\lag$ and assume the image of the classifier is uniquely determined by $\theta$. Obviously the distribution of classifier outputs is not the true distribution of labels, since $y$ is fixed. {\it Conversely}, however, if we fix a classifier and can measure its error empirically, then we may ask for the most probable distribution of possible $Y$ values given (i) a value $X = x$ and (ii) our knowledge that $h_\theta(x)$ has some conditionally expected error $\E[\lag(h_\theta(x), y) \mid x] = \xi$. Furthermore, if we know the minimal error $\xi^*$, we can then compare that distribution to the distribution we would obtain if $h_\theta(x)$ were the Bayes classifier so that $\xi = \xi^*$, using the relative entropy or the distance between $\xi$ and $\xi^*$ to estimate how far $h_\theta$ is from $g$ (indeed these should be the same procedures in some suitable sense). That is to say: on one hand, we can determine what the true distribution most likely is if we know $\xi$. On the other, we can ask what the statistics of the distribution would need to be for a fixed $h_\theta(x)$ to be its Bayes classifier. Together, we can determine the minimum possible divergence between the true posterior and any inferred posterior given $h_\theta(x)$, and how that divergence varies with respect to changing $\theta$. 


From this we claim finding the Bayes optimal classifier empirically involves inferring the true posterior and then studying large deviations in $P_\theta(Y \mid X)$ constrained by an empirically expected error $\xi$. By Sanov's theorem, the probability of deviations in empirical measures from a true measure are dominated by their closest points to that true measure, measured by the relative entropy. Likewise it is known that under certain {\it desiderata} on an inference (outlined in \cite{shore1980axiomatic} and reviewed in the following section) the entropy
\[
    - \int P \log P - \lambda \left( \int V P - C\right)
\]
is singled out as the functional to be extremised, where given a reference measure $Q$, $V$ is any measurable function with growth such that $e^{-\lambda_C V}$ is locally $L^1$ with respect to $Q$, for some value of $\lambda$ determined by $C$. 

We will here discuss a collection of results linking the optimality of the Bayes classifier with the optimality of the principle of maximum entropy as a form of constrained inference. In particular, our proofs will use the theory of large deviations to link constrained probabilities with {\it maximum a posteriori} estimates. The generality of our framework goes well beyond quadratic utility functions, addressing {\it e.g.} a critique of Berger concerning the use of mean squared error. On the other hand, we will only consider unimodal distributions; that is, we will assume $\lag$ is lower semi-continuous, convex, and at least twice continuously differentiable. We will assume the extremal point of the relative entropy is also unique. Necessary and sufficient conditions for this assumption to hold are given in \cite{csiszar}. We will finally implicitly assume absolute continuity wherever necessary, {\it e.g.} in discussions of the relative entropy. We will also assume the true measure and true labels are known, say, for the purpose of training or fitting parametric distributions to the true measure.

We are grateful to K Dill and K Friston for many useful discussions around this topic. The author acknowledges support from the VERSES Research Lab and the Einstein Chair programme at the Graduate Centre of the City University of New York.

\section{A review of large deviations and maximum entropy principles}\label{review-sec}

In this section we will review large deviations theory and the principle of maximum entropy in the setting of large deviations. Excellent references include \cite{dembo2009large, varadhan2016large} and \cite[\S 24]{kallenberg1997foundations}. Throughout we will abbreviate formulae by neglecting to write base measures when doing so is inessential. 

Suppose we have a random variable parametrised by some quantity $n$. It is natural to ask how the probability of any set varies with $n$, and to try and obtain asymptotic statements about $P_n(X \in A)$ as $n \to \infty$ or $n \to 0$. Let $c_n$ be a sequence of real numbers which diverges as $n\to 0$, and let $I$ be a lower semi-continuous function on $M$. The sequence $\{P_n\}$ satisfies a large deviations principle with rate function $I$ if and only if, for all Borel sets $A$, the probability of fluctuating away from the minimiser of $I(A)$ decays exponentially in $n$, to leading order in $n$. In particular, for $A^\circ$ the interior of $A$ and $\bar{A}$ its closure, we have
\[
    - \inf_{x \in A^\circ} I(x) \leqslant \liminf_{n\to0} \frac{1}{c_n} \log P_n(X \in A) \leqslant \limsup_{n\to0} \frac{1}{c_n} \log P_n(X \in A) \leqslant -\inf_{x\in\bar{A}} I(x),
\]
or more concisely (if $A$ is a closed convex set), 
\[
    \lim_{n\to0} \frac{1}{c_n} \log P_n(X \in A) = - \inf_{x\in A} I(x).
\]
In effect this states that the dominant behaviour of the probability that $X$ takes a value near some $x$ is $\exp(-c_n I(x))$. Moreover, the error is sub-exponential in $n$ (including in the tails), and that in the limit $c_n \to \infty$, the measure concentrates on $x^* = \argmin I(x)$. Hence a large deviations principle describes a situation where the probability of observing a value of a random variable which is `far' from a typical value decays exponentially to first order in some parameter $n$. We denote the above using the shorthand
\[
    P(X \in A) \asymp e^{-c_n \inf_{a \in A} I(a)}
\]
to suggest we are taking the Laplace approximation of the mass of $P$ by its dominant points (see \cite[\S E7]{touchette2009large} for example). 

If the law of $X$ satisfies a large deviations principle with rate function $I$ then any measurable function $f$ of $X$ defining a random variable $Y$ has a pushforward measure, also satisfying a large deviations principle with rate $J$, obtained as 
\[
    J(y) \coloneqq \inf_x \{ I(x) \mid f(x) = y\}
\]
by using the Laplace approximation of the pushforward measure. We can interpret this two ways: (i) thinking of the rate function as a measurement of improbability, the probability of $y$ is identically the probability of the least improbable $x$'s such that $f(x) = y$; or, (ii) the satisfactory $x$'s are the most probable atypical events given that $f(x) = y$ (a non-empty subset of the $\sigma$-algebra of $\Omega$ by assumption). On the subject of how (i) relates to (ii), more can be said. In the primal form, we can express the satisfactory $x$'s as solving an unconstrained variational problem
\[
    \inf_{x\in A} \sup_\beta I(x) + \beta(f(x) - y)
\]
where $\beta \in \R$ is a Lagrange multiplier enforcing the constraint, the minimised quantity having the expression 
\[
    \sup_\beta I(x) + \beta(f(x) - y) = 
        \begin{cases}
            I(x) & \text{if } f(x) = y\\
            \infty & \text{otherwise}
        \end{cases}
\]
from choosing $\beta = \infty$ when $f(x) - y > 0$ such that the supremum is infinite ($\beta = -\infty$ when $f(x) - y < 0$, respectively). Now observe that, from Bayes' theorem,
\[
    P(X \in A \mid f(X) = y) = \frac{P(A \cap \{x: f(x) = y\})}{P(\{x: f(x) = y\})}.
\]
The most probable $x$ in $A$ hence minimise, if $f(x) = y$, the total rate function
\[
    I_y(x) \coloneqq \inf_{x \in A \cap \{x:f(x) = y\}} I(x) - \inf_{\{x:f(x) = y\}} I(x) \geqslant 0.
\]
One sees that $\inf_{x \in A} I_y(x)$ occurs when the least costly $x \in A$ such that $f(x) = y$ is the least costly $x$ such that $f(x) = y$; in other words, writing 
\[
    \inf_{x \in A \cap \{x:f(x) = y\}} I(x) \geqslant \inf_{\{x:f(x) = y\}} I(x),
\]
the infimum of $I_y(x)$ is attained when we do no better on the parts of $\{x : f(x) = y\}$ not contained in $A$ than we do on $A \cap \{x : f(x) = y\}$. Note that for those parts of $A$ not contained in $\{x : f(x) = y\}$, $I(x)$ is infinite, such that $I_y(x)$ diverges by construction---so there is no way to improve the rate function by leaving $\{x : f(x) = y\}$. Indeed, as established above, the primal rate function is equivalent to a rate function on a constrained support. It follows that 
\[
    \inf_{x \in A} I_y(x) = \inf_{x \in A \cap \{x:f(x) = y\}} I(x)
\]
and therefore
\[
    J(y) = \inf_{x \in A} I_y(x).
\]
Hence the probability of $y$ is determined by the least improbable $x$ in $A$ such that $f(x) = y$. We will use these facts in \S\ref{max-ent-uncertainty-section}---specifically in the proof of Theorem \ref{max-likelihood-thm}.

Contraction is of particular interest when we wish to study how large deviations in the statistics of random measures are determined by large deviations in those measures. Let $L_n$ be the empirical vector of $n$ independent observations of identically distributed random variables $X$ valued in a Polish space, and $V \cdot L_n$ be the empirical expectation
\[
    \int V(x) L_n(x) \dd x.
\]
It is known that (under suitable conditions) the law of $L_n$ satisfies a large deviations principle in the space of measures (itself a Polish space) equipped with the topology of weak convergence, given by
\begin{equation}\label{ldp-sanov-eq}
    \lim_{n\to\infty} \frac{1}{n} \log P^n(L_n \in K) = -\inf_{L_n \in K} D_{\mathrm{KL}}(L_n \| P)
\end{equation}
where $P$ is the law of $X$; that is to say, it has log-probability equal (to first order in $n$) to the negation of the minimal possible divergence between $P$ and $K$ in the space of measures \cite{sanov}. Given the rate function for $L_n$ we can also consider the probability of $L_n$ conditioned to lie in a `rare set' with certain statistics $V \cdot L_n = c$. Contracting \eqref{ldp-sanov-eq} to a large deviations principle for empirical expectations entails taking
\[
    E(c) = \inf_{\mu : V\cdot \mu = c} D_{\mathrm{KL}}(\mu \| P)
\]
thus expressing $\inf_c E$ (the rate function for averages of $V$) in terms of the rate function for $L_n$. The $L_n$ solving this constrained optimisation problem minimise the primal function 
\begin{equation}\label{global-rate-eq}
\sup_\beta \big( D_{\mathrm{KL}}(\mu \| P) - \beta(V\cdot \mu - c) \big) = 
    \begin{cases}
        D_{\mathrm{KL}}(\mu \| P) & \text{if } V \cdot \mu = c\\
        \infty & \text{otherwise}
    \end{cases}
\end{equation}
and by convex duality, we can equivalently solve
\[
    \sup_\beta \inf_\mu\big( D_{\mathrm{KL}}(\mu \| P) - \beta(V\cdot \mu - c) \big),
\]
which is the free energy. It is possible to prove that the constraint is equivalent to contracting to the Cram\'er rate function for empirical means, such that we need to minimise
\[
D_{\mathrm{KL}}(\mu \| P) - \inf_{\mu : V\cdot \mu = c} D_{\mathrm{KL}}(\mu \| P)
\]
occurring when the KL divergence of $\mu$ is equivalent to the intended expectation of $V$. The takeaway in either case is that $L_n$ converges to the measure which minimises the relative entropy subject to our constraint \cite{csiszar}. That is to say, if we consider an improbable set given by the constraints, then the {\it least improbable} improbable measure is the closest one (in information) to $P$, with probability the exponential of that `distance'.

This observation is of great importance in statistical mechanics, where it describes the asymptotic behaviour of functions of states of a system in the so-called thermodynamic limit \cite{ellis2006entropy, touchette2009large}. 
One way to think of statistical mechanics is as a method for reconstructing laws of thermodynamics from microscopic mechanical laws suitably `averaged'. Another is that in statistical mechanics one is driven by a problem of inference. Suppose we have a system of $k$ interacting particles described by position states in three-space. The configuration of the system is a $3k$-dimensional vector called a microstate of the system. In thermodynamics we can measure coarse-grained variables describing the ensemble of particles, such as temperature, pressure, and volume (macrostates of the system), but do not generally have access to the microstates when $k$ is large, or do not want the microstates if we are modelling the macrostates directly---say, for phenomenological reasons, such as modelling the results of laboratory experiment. On this view, the {\it ethos} of statistical mechanics is to infer the most probable microstates given a measurement of the average value of a macrostate. Indeed, such large deviations principles are called `Gibbs conditioning principles' in \cite{dembo1996refinements}.

The philosophy described previously is present most strongly in the work of Jaynes \cite{jaynes1957information, jaynes1980minimum} who proposed that finding an equilibrium state by maximising entropy should be thought of as parametrising the assignment of probabilities to microstates by macrostate values, whilst preserving all other randomness coming from our uncertainty about the system (see \cite{ken} for a review). Impressionistically, notice that \eqref{global-rate-eq} vanishes when the difference in information between $\mu$ and $P$---representing uncertainty we want to be resolvable with knowledge of the constraints---is equivalent to knowledge of the mean of $V$. Physically, the free energy (entropy constrained by an average of some observable) is a variational functional of probabilities of macrostates whose systematic variance with the distance of the system from equilibrium allows us to infer the most likely distribution of states, and in particular, allows us to make that parametrisation by considering the likely macrostate values of an equilibrium state, whilst also considering the volume of the system in phase space---our uncertainty about it. The principle of maximum entropy can thus be thought of as the least biassed method for consistently performing a constrained inference---axiomatised as such by Shore--Johnson \cite{shore1980axiomatic, presse2013nonadditive}. 

One should note that the content of Jaynes' work is much more general than statistical physics; the canonical nature of entropy as a form of measurement of information was arrived at by Shannon \cite{shannon}, Kullback--Leibler \cite{kullback1951information}, and Khintchine \cite{aleksandr1957mathematical} using purely probabilistic arguments; Jaynes and his contemporaries were aware of this, as discussed in {\it e.g.} \cite{jaynes1980minimum, williams1980bayesian, skyrms1987updating, jaynes1988relation, csiszar1991least, jaynes2003probability}. However, for physical conditions the thermodynamic entropy is proportional to the Shannon entropy (with constant of proportionality equal to Boltzmann's constant) \cite{jizba2020shannon}. A cogent discussion of this point can be found in \cite{presse}, and the observation there that the Shannon entropy is {\it a priori} independent of any statement about physics---whereas obtaining from it a thermodynamic entropy depends on the parameters of the thermodynamic state of the system being defined---goes back to \cite{jaynes1965gibbs}.

\section{Maximum entropy is Bayes optimal}

If we seek the most likely measure consistent with an estimated error value, that is, the most likely assignment of probabilities to classes such that the given error is observed, we will find the following.

\begin{theorem}\label{max-ent-bayes-thm}
    Empirically classifying posteriors associated to prescribed errors on predicted states is Bayes optimal in the limit of infinitely many samples if and only if the error function on predicted measures is the relative entropy.
\end{theorem} 
\begin{proof}
    Fix $h_\theta$, a reference measure $q$, and an error function $\lag$ on predicted states. Suppose we can empirically estimate the conditionally expected error $\xi\coloneqq L_n \cdot \lag(h_\theta(x), y)$. Let $\chi \colon \xi \to p(y \mid x)$ be a classifier mapping errors to random vectors. Define the measure $Q(p(y \mid x) \mid L_n \cdot \lag(h_\theta(x), y) = \xi)$. Let $\G(\chi(\xi), q)$ be the loss function on predicted measures. In the limit $n \to \infty$, by Gibbs conditioning, $Q \to \delta\left(e^{-\lambda_\xi \lag(h_\theta(x), y)}q \right)$. The expected loss on measures is then 
    \[
        \G(e^{-\lambda_\xi \lag(h_\theta(x), y)}q, q).
    \]
    The classifier is fixed by the limit. To minimise this expression, let $\G$ be the relative entropy. This proves sufficiency. To prove necessity, set $\G = \int f(p) p \dd{x}$ for $f$ an arbitrary functional of measures satisfying the Weierstrass conditions and any $V$. The variational derivative of $\G$ is 
    \[
        \partial_p \big( f(p) p\big) = f(p) + p \partial_p f(p).
    \]
    Now we solve the resulting differential equation for extremal values of $\partial_p \big( f(p) p\big)$ at the point $p = \exp(-\lambda V)q$:
    \[
        e^{-\lambda V(x)}q f' \left(e^{-\lambda V(x)}q \right) = - f\left(e^{-\lambda V(x)}q \right).
    \]
    By direct inspection $f(p) = \log p/q\, + \lambda V - 1$ is a nontrivial solution and by continuity this solution is unique.
\end{proof}

Theorem \ref{max-ent-bayes-thm} implies that any other choice of $f$ will lead to an increase in risk. It is shown by Shore--Johnson \cite{shore1980axiomatic} and Khintchine \cite{aleksandr1957mathematical} that the form of $f(p)$ is uniquely determined by four axioms for making unbiassed inferences. In particular any other choice of $f$ will produce spurious correlations between independent samples of a random variable \cite{presse2013nonadditive} and contradict Bayes' theorem as a result \cite{presse2014nonadditive} (whereas maximum entropy is consistent with Bayes' rule \cite{jaynes1988relation}). Indeed, obtaining the principle that entropy is maximised at equilibrium from the assumption that subsystems are independent unless proven otherwise goes back to Gibbs' original argument\footnote{The reader may wish to consult \cite[\S 10]{emch2013logic} for a contemporary review. This argument is also laid out in detail in \cite{jaynes1965gibbs}.} that his expression for the entropy reproduces the fundamental thermodynamic relation of two systems in equilibrium with each other such that it is compatible with the thermodynamic properties of extensive variables under divisions of a system \cite[\S XIV]{gibbs1902elementary}.

One can imagine that by assuming a structure for correlations of outputs which is not implied by the data, we assign probabilities to classes inconsistently with the data observed with a false sense of certainty, so that changing $\G$ artificially decreases risk---in particular, in a way incompatible with the measured conditionally expected error $\xi$.

In fact we can prove that such biasses decrease the risk of misclassification when these correlations sharpen the peak of the distribution. To control the rate of decay of the measure and relate it to the correlation of $X$ and $Y$ we will envelop the measure with a Gaussian and assume its variance controls all higher moments, fix the difference between the variance of the measure and that of the enveloping Gaussian, and study the asymptotics of the Gaussian.

\begin{theorem}\label{legendary-bound-thm}
    Suppose there exists a constant $\sigma \geqslant 0$ such that for all $p > 0$ the inequality 
    \[
        \E[Y^p \mid X] \leqslant \left(\sigma^2 p\right)^{p/2}
    \]
    is satisfied. Let $r \in [0,1]$ be the correlation between $X$ and $Y$ and $\eps \geqslant 0$ a constant depending on $\sigma^2$. For all $\eps$ there exist constants $k, c \geqslant 0$ such that 
    \[
        \E[\lag(h_\theta(x), y) \mid X] \approx kc(1 - r^2) - k\eps
    \]
    with error of order $o\!\left(3\sigma^3\sqrt{3}\right)$. Consequently, $\xi$ decreases with increasing correlation.
\end{theorem}
\begin{proof}
    For brevity we will denote the conditional expectation of $Y$ given $X$ as $m$. By bounding the moment generating function of $Y$ our hypothesis implies that $P(Y \mid X)$ is majorised by a normally distributed random variable with conditional variance $\sigma^2$. Then we have 
    \[
        \V[Y \mid X] \leqslant \sigma^2_{Y \mid X} = \sigma^2_Y(1-r^2).
    \]
    Let $\eps = \sigma^2_{Y\mid X} - \V[Y \mid X]$. Now take the Taylor expansion in a neighbourhood of $m$ to obtain
    \[
        \E[\lag(h_\theta(x), y) \mid X] = \lag(h_\theta(x), m) + \frac{\partial_{y}^2 \lag(h_\theta(x), m)}{2} \big(\sigma^2_Y(1-r^2) - \eps\big) + o(\E[(Y - m)^3 \mid X]).
    \]
    Since $\lag$ is a convex function by assumption, its second derivative is always non-negative. For fixed $\eps, \sigma^2_Y$, as $r$ increases, the conditional expectation of $\lag$ decreases.
\end{proof}

\begin{remark}
    Note that these results are complementary: in Theorem \ref{max-ent-bayes-thm}, we have shown that given an error, we `ought to' use the method of maximum entropy to infer the distribution creating precisely that error, whilst in Theorem \ref{legendary-bound-thm} we have shown that if we use some other prescription, we will make an inference with a lower error when that prescription produces correlations not implied by the data.
\end{remark}

Similarly, the probability of an empirical distribution with less information than we already have (in the precise sense of having a greater entropy) is exponentially small. The next two results will establish this. 

\begin{lemma}\label{min-lem}
    For any given continuous loss function, the distribution globally minimising the relative entropy away from $P(Y \mid X)$ attains the minimal conditionally expected error when $h_\theta(x)$ is fixed to the Bayes classifier of $P(Y \mid X)$. Conversely, if a distribution has the same Bayes classifier as $P(Y \mid X)$ for any continuous loss function, it globally minimises the relative entropy away from $P(Y \mid X)$. 
\end{lemma}
\begin{proof}
    Recall that $D_{\mathrm{KL}}(F \| P) = 0$ if and only if $f \overset{F\text{-p.p.}}{=} p$. Almost sure equality of measures implies 
    \begin{equation}\label{integral-eq}
        \int_N \lag(h_\theta(x), y) f(y \mid x) \dd y = \int_N \lag(h_\theta(x), y) p(y \mid x)\dd y.
    \end{equation}
    Setting $h_\theta(x)$ to the Bayes classifier proves the first part of the claim. To prove the second, suppose \eqref{integral-eq} holds for all continuous $\lag$. By uniqueness of Riesz representations, $f \overset{F\text{-p.p.}}{=} p$.
\end{proof}

\begin{prop}\label{contraction-prop}
    Let $\Pi$ be the set of measures with expected error $\xi$. To first order in $n$, any atypical error value $\xi \in \Xi$ occurs with probability $\exp\big(-n \inf_\Xi \inf_\Pi D_{\mathrm{KL}}(\Pi \| P)\big)$. Assume the complete class theorem holds. With probability one, in the limit of infinitely many samples, $L_n$ converges to a measure in the measure class for which $h_\theta$ is the Bayes classifier. 
\end{prop}
\begin{proof}
    By contraction we have the rate function for errors
    \[
        I(\xi) = \inf_{\mu : \lag\mu = \xi} D_{\mathrm{KL}}(L_n \| P)
    \]
    so that the negative log-probability of an error value is (to first order in $n$) the relative entropy of the minimal entropy empirical vector for which $L_n \cdot \lag(h_\theta, y) = \xi$. As such, 
    \[
        \frac{1}{n} \log P(\xi \in \Xi) \asymp - \inf_{\xi \in \Xi}\left( \inf_{\mu : \lag\mu = \xi} D_{\mathrm{KL}}(L_n \| P)\right).
    \]
    For the remainder of the claim, we must prove a solution to this optimisation problem exists. By the complete class theorem there exists a Bayes classifier $h_\theta(x)$ for $P$ and $\lag$ such that $\lag \cdot P = \xi^*$. It follows from Lemma \ref{min-lem} that if $D_{\mathrm{KL}}(\mu \| P) = 0$ for this choice of $h_\theta$ then $\xi$ is minimised, so that the infimum is attained as $n\to\infty$.
\end{proof}

Consequently, for any $n$, the most likely measure $L_n$ is the closest measure to $P$ for which the given classifier is the Bayes classifier, converging to that measure exponentially quickly in sample size. Shannon proved that the entropy is the amount of information it would take to resolve our uncertainty about a distribution in the celebrated paper \cite{shannon}. The above result then states that overlooking available information about what the data looks like does just as poorly as building in extraneous assumptions.

These results give us a general characterisation of such estimated distributions: in order to be close in $I$-projection to $P$, the probability of any label must decay exponentially in loss away from the output of the classifier. That is---given a mode, in the absence of further information, the rest of the distribution is most likely to be exponential away from that mode, forcing the tails to fall off rapidly. To first order in $n$, $\exp(-n\lag(h_\theta(x), y))$ is a good approximation of $P$.

\begin{corollary}\label{exp-cor}
    For a fixed error function and classifier, the probability of a label decays exponentially, such that the mode is the label equal to the classifier output. Consequently contraction on the rate function of errors yields a rate function for classifiers whose mode is the image of the Bayes classifier of $P$.
\end{corollary}

The distribution of minimal relative entropy enjoys an additional interpretation as the distribution with the least complexity, in the following sense. The information gain measures how much information about classes is contained in an observation, in that the information gained is the {\it reduction} in the information required to resolve our uncertainty about one random variable upon observing another random variable---in other words, a reduction in the Shannon entropy. By minimising the relative entropy we use only the information available to us in the reference distribution. We will hence interpret the minimisation of relative entropy as a {\it law of parsimony}. Namely, the best way to model an unlikely event is the simplest way given the information we already have, with the least additional information built in. Limiting such assumptions, which are confirmed or denied by further observations, keeps the expected surprise minimal and minimises the risk of misclassification. 

Following the set up in \cite{jaynes1988relation}, we will let $\mu = P(Y \mid X)$ be a model of classes and consider a large deviation in observations of $Y$, such that $L_n$ lies in a rare set with mode not equal to that of $P(Y \mid X)$. We may conclude that $L_n \cdot \lag(h_\theta(x), y) > \xi^*$. The next corollary follows from Proposition \ref{contraction-prop}; namely, that the measure on any set of models consistent with $h(\mu) = \xi$ is dominated by the probability of the minimal entropy model.

\begin{corollary}\label{models-cor}
    The posterior with minimal risk of misclassification is the one with the least complexity.
\end{corollary}

The domination of the measure by the most probable of the possible improbable events is a classic aphorism in large deviations, and is its own form of Ockham's razor. 

\begin{remark}
These large deviations estimates for Bayes classifiers are {\it maximum a posteriori} estimates in a very natural way, since as the speed parameter diverges, the rate function becomes a zero-one loss.
\end{remark}

\section{Classification in the presence of known unknowns}\label{max-ent-uncertainty-section}

In this section we will consider a Bayes classifier where the loss on classes is measured with some uncertainty. Given a random variable of expected losses
\[
    L_n \cdot \lag(h_\theta(x), y) \in \Xi
\]
with a reference distribution $Q$ and an empirical distribution over $\Xi$, $M_k$, our task is to infer the likeliest distribution of models $\mu$ given knowledge of 
\[
    U(L_n \cdot \lag(h_\theta(x), y))\cdot M_k = \eta 
\]
for some $U, \eta$. Since one knows $M_k$ satisfies a large deviations principle, we will constrain the large deviations principle $P^n(\mu)$ satisfies with knowledge of the large deviations principle $Q^k(\nu)$ satisfies. To do so, we will firstly need a technical result. From now onwards, we denote by $\{\mu\}_\Xi$ the set $\{\mu : \lag \cdot \mu \in \Xi\}$.

\begin{theorem}\label{level-three-thm}
    There exists a level three large deviations principle associated to the empirical conditional probabilities of empirical measures with prescribed means 
    \[
        M_k = P^n(\{\mu : \lag\cdot\mu \in \Xi\} \mid \mu \in K),
    \]
    with rate function 
    \[
        -\inf_{\nu \in H} D_{\mathrm{KL}}(\nu \| Q).
    \]
\end{theorem}
\begin{proof}
    For $M$ a Polish space, $\mathcal{P}(M)$ is Polish, and therefore $\mathcal{P}(\mathcal{P}(M))$ is Polish. The result then follows from Sanov's theorem applied to measures of conditional measures, whence 
    \[
        \frac{1}{k} \log Q^k(M_k \in H) \asymp - \inf_{\nu \in H} D_{\mathrm{KL}}(\nu \| Q)
    \]
    with the most probable conditional distribution over empirical distributions being $\nu^*$. Consider the contraction to the probability of empirical expectations of a function $U\colon \{\mu\}_\Xi \to \{\mu\}_\Xi$ under $M_k$. Denote this expected empirical measure as $U(\mu)\cdot\nu(\mu) = \eta$. There is a large deviations principle
    \[
        \frac{1}{k} \log P^k(\eta \in N) \asymp - \inf_{\eta \in N} \inf_{\{\nu : U \cdot\nu = \eta\}} D_{\mathrm{KL}}(\nu \| Q)
    \]
    making $D_{\mathrm{KL}}(\nu \| Q)$ a level three rate function. 
\end{proof}

The interpretation of this theorem is that the conditional measures in question must be estimated as models of the statistics of $\mu$, given our uncertainty about how expectations $\lag\cdot \mu$ are distributed in $\Xi$. It is a level three large deviations principle in the sense that it describes the empirical process of measuring $\lag\cdot\mu$ and estimating the most likely distribution of errors. This interpretation will become clearer in the following discussion.

\begin{theorem}\label{max-likelihood-thm}
    Computing the most probable empirical measure given some range of constraint values is equivalent to a certain maximum likelihood problem.
\end{theorem}
\begin{proof}
    We will begin by recalling that
    \[
        P^n(\mu \in K \mid \lag\cdot\mu \in \Xi) \propto P^n(\{\mu : \lag\cdot\mu \in \Xi\} \mid \mu \in K) P^n(\mu \in K).
    \]
    Theorem \ref{level-three-thm} gives us an explicit expression for the optimal estimate of the first factor by putting
    \[
        \nu^* = \argmin_{\{\nu : U\nu = \eta\}} D_{\mathrm{KL}}\Big(\nu(\{\mu\}_\Xi \mid \mu \in K) \| Q(\{\mu\}_\Xi \mid \mu \in K)\Big).
    \]
    Solving the convex dual problem to obtain an expression for the minimiser yields (denoting the value of the satisfactory Lagrange multiplier as $\lambda_\eta$)
    \[
        \nu^* = e^{-\lambda_\eta U(\{\mu\}_\Xi)}Q(\{\mu\}_\Xi)
    \]
    from which we have 
    \begin{align*}
        \frac{1}{n}\log &P^n(\mu \in K \mid \lag \cdot \mu \in \Xi, U\cdot \nu = \eta) \\ &\asymp -\inf_{\mu \in K}\left(\inf_{\mu \in K \cap \{\mu\}_\Xi}D_{\mathrm{KL}}(\mu \| P) + 
        \lambda_\eta U(\{\mu\}_\Xi) - \log Q(\{\mu\}_\Xi) \right).
    \end{align*}
    Using additivity and simplifying all three terms, we arrive at
    \[
        -\inf_{\mu \in K \cap \{\mu\}_\Xi} D_{\mathrm{KL}}(\mu \| P) -\inf_{\mu \in K \cap \{\mu\}_\Xi}\lambda_\eta U(\mu) + \sup_{\mu \in K \cap \{\mu\}_\Xi} Q(\mu).
    \]
    It follows that we have converted our study of the point of concentration to the easily soluble maximum likelihood estimate
    \begin{equation}\label{max-likelihood-eq}
        \argmax_{\mu \in K \cap \{\mu\}_\Xi} e^{- D_{\mathrm{KL}}(\mu \| P)} e^{-\lambda_\eta U(\mu)}Q(\mu).
    \end{equation}
\end{proof}

Lastly, we have 

\pagebreak

\begin{prop}
    The maximum of \eqref{max-likelihood-eq} is the model $\mu = P(Y \mid X)$ whose most probable output is the image of the Bayes classifier. 
\end{prop}
\begin{proof}
    This is immediate by Theorem \ref{level-three-thm} followed by Proposition \ref{contraction-prop}---namely, contracting the level three large deviations principle on empirical conditional probabilities of empirical measures down to the level one large deviations principle of empirical means. In particular, the most likely $\mu$ minimises the total rate function
    \[
        -\inf_{\mu \in K \cap \{\mu\}_\Xi} D_{\mathrm{KL}}(\mu \| P) -\inf_{\mu \in K \cap \{\mu\}_\Xi}\lambda_\eta U(\mu) + \sup_{\mu \in K \cap \{\mu\}_\Xi} Q(\mu).
    \]
    By Proposition \ref{contraction-prop}, in the limit, the minimal $\xi$ in $\Xi$ is selected. 
\end{proof}

We have obtained a distribution over constraint values by maximising entropy to characterise our known unknowns, and then seek to maximise entropy again on the prior factor to get our most probable model. Letting $\xi$ be an observation of an error value, we are considering now the Bayesian inverse of $P^n(\mu \mid \xi)$, where the first factor of \eqref{max-likelihood-eq} is a prior over models and the second factor is the probability of an observed error value given a model, inferred by maximising the entropy of the observation distribution as a function of $\mu$. The following corollary can then be stated: 

\begin{corollary}
    The distribution $P^n$ has the interpretation of a Bayesian posterior over models given observations. 
\end{corollary}

An example is useful: 

\begin{example}
    Let $U(\{\mu\}_\Xi) = (\lag\mu - \E_\nu[\lag\mu])^2$. The final result is 
    \[
    e^{-D_{\mathrm{KL}}(\mu \| P)}e^{-\lambda_\eta(\lag\mu - \E_\nu[\lag\mu])^2} 
    \]
    where $\eta$ is now the variance in measurements of expected losses. The second factor is the distribution of models whose expected loss is in $\Xi$ given their location in $K$ and our knowledge of the constraints satisfied by the distribution over expected losses. As such, we have weighted our estimate of the probability of any $\mu$ by a {\it test misfit statistic}, measuring the deviation (in squared Euclidean distance) of loss expected under $\mu$ from the measured average of expected losses. 
\end{example}

Much more can be said about this prescription and we hope to investigate it further in the future.

\section{Concluding remarks}

In this paper we have discussed the Bayes optimality of maximum entropy inference, in the sense that no extraneous information is included in the inference performed. The takeaway of this result is firmly in the spirit of Jaynes' and even Gibbs' original arguments about entropy as a best approximation to some unknowns given known quantities. Namely---if we imagine making inferences which avoid undue risk by finding the empirical distribution describing a large number of samples from an arbitrary distribution satisfying certain expectation values, then the empirical distribution will (with large probability) be that which maximises the Shannon entropy. Any assumptions beyond what is sampled are unwarranted if we want to avoid risking misclassification. 

It is important to note that this does not guarantee an optimal inference is performed, say, if inadequate information is contained. This has been explored in the context of Bayesian inference and the information contained in prior distributions in \cite{wolpert1996reconciling} and in the context of inadequately constrained models in papers such as \cite{presse, agozzino2019minimal}.

\bibliographystyle{amsalpha}
\bibliography{main}

\providecommand{\bysame}{\leavevmode\hbox to3em{\hrulefill}\thinspace}
\providecommand{\MR}{\relax\ifhmode\unskip\space\fi MR }
\providecommand{\MRhref}[2]{%
  \href{http://www.ams.org/mathscinet-getitem?mr=#1}{#2}
}
\providecommand{\href}[2]{#2}
\begin{thebibliography}{PGLD13b}

\bibitem[AD19]{agozzino2019minimal}
Luca Agozzino and Ken Dill, \emph{Minimal constraints for maximum caliber analysis of dissipative steady-state systems}, Physical Review E \textbf{100} (2019), no.~1, 010105.

\bibitem[Ber85]{berger2013statistical}
James~O Berger, \emph{Statistical decision theory and {B}ayesian analysis}, Springer Series in Statistics, Springer, 1985.

\bibitem[Csi84]{csiszar}
Imre Csisz{\'a}r, \emph{Sanov property, generalized {I}-projection and a conditional limit theorem}, The Annals of Probability \textbf{12} (1984), no.~3, 768--793.

\bibitem[Csi91]{csiszar1991least}
\bysame, \emph{Why least squares and maximum entropy? {A}n axiomatic approach to inference for linear inverse problems}, The Annals of Statistics \textbf{19} (1991), no.~4, 2032--2066.

\bibitem[DZ96]{dembo1996refinements}
Amir Dembo and Ofer Zeitouni, \emph{Refinements of the {G}ibbs conditioning principle}, Probability Theory and Related Fields \textbf{104} (1996), no.~1, 1--14.

\bibitem[DZ10]{dembo2009large}
Amir Dembo and Ofer Zeitouni, \emph{Large deviations techniques and applications}, 2nd ed., Stochastic Modelling and Applied Probability, vol.~38, Springer, 2010.

\bibitem[EL02]{emch2013logic}
Gerard~G Emch and Chuang Liu, \emph{The logic of thermostatistical physics}, Springer, 2002.

\bibitem[Ell85]{ellis2006entropy}
Richard~S Ellis, \emph{Entropy, large deviations, and statistical mechanics}, Grundlehren der mathematischen Wissenschaften, vol. 271, Springer, 1985, 2006 ``Classics in Mathematics'' reprint.

\bibitem[Gib02]{gibbs1902elementary}
Josiah~W Gibbs, \emph{Elementary principles in statistical mechanics}, Charles Scribner's Sons, 1902.

\bibitem[Jay57]{jaynes1957information}
Edwin~T Jaynes, \emph{Information theory and statistical mechanics}, Physical Review \textbf{106} (1957), no.~4, 620--630.

\bibitem[Jay65]{jaynes1965gibbs}
\bysame, \emph{Gibbs vs {B}oltzmann entropies}, American Journal of Physics \textbf{33} (1965), no.~5, 391--398.

\bibitem[Jay80]{jaynes1980minimum}
\bysame, \emph{The minimum entropy production principle}, Annual Review of Physical Chemistry \textbf{31} (1980), no.~1, 579--601.

\bibitem[Jay88]{jaynes1988relation}
\bysame, \emph{The relation of {B}ayesian and maximum entropy methods}, Maximum-Entropy and Bayesian Methods in Science and Engineering, vol.~1, Springer, 1988, pp.~25--29.

\bibitem[Jay03]{jaynes2003probability}
\bysame, \emph{Probability theory: The logic of science}, Cambridge University Press, 2003.

\bibitem[JK20]{jizba2020shannon}
Petr Jizba and Jan Korbel, \emph{When {S}hannon and {K}hinchin meet {S}hore and {J}ohnson: equivalence of information theory and statistical inference axiomatics}, Physical Review E \textbf{101} (2020), no.~4.

\bibitem[Kal21]{kallenberg1997foundations}
Olav Kallenberg, \emph{Foundations of modern probability, volume 2}, 3rd ed., Probability Theory and Stochastic Modelling, Springer, 2021.

\bibitem[Khi57]{aleksandr1957mathematical}
Aleksandr~Ya Khintchine, \emph{Mathematical foundations of information theory}, Courier Corporation, 1957.

\bibitem[KL51]{kullback1951information}
Solomon Kullback and Richard~A Leibler, \emph{On information and sufficiency}, The Annals of Mathematical Statistics \textbf{22} (1951), no.~1, 79--86.

\bibitem[Mac03]{mackay2003information}
Sir David J~C MacKay, \emph{Information theory, inference and learning algorithms}, Cambridge University Press, 2003.

\bibitem[PGLD13a]{presse2013nonadditive}
Steve Press\'e, Kingshuk Ghosh, Julian Lee, and Ken~A Dill, \emph{Nonadditive entropies yield probability distributions with biases not warranted by the data}, Physical Review Letters \textbf{111} (2013), no.~18, 180604.

\bibitem[PGLD13b]{ken}
\bysame, \emph{Principles of maximum entropy and maximum caliber in statistical physics}, Reviews of Modern Physics \textbf{85} (2013), no.~3, 1115--1141.

\bibitem[PGLD15]{presse}
\bysame, \emph{{Reply to C. Tsallis' ``Conceptual inadequacy of the Shore and Johnson axioms for wide classes of complex systems''}}, Entropy \textbf{17} (2015), no.~7, 5043--5046.

\bibitem[Pre14]{presse2014nonadditive}
Steve Press{\'e}, \emph{Nonadditive entropy maximization is inconsistent with {B}ayesian updating}, Physical Review E \textbf{90} (2014), no.~5, 052149.

\bibitem[San57]{sanov}
Ivan~N Sanov, \emph{On the probability of large deviations of random variables}, Matematicheskii Sbornik Novaya Seriya \textbf{84} (1957), no.~1, 11--44.

\bibitem[Sha48]{shannon}
Claude~E Shannon, \emph{A mathematical theory of communication}, The Bell System Technical Journal \textbf{27} (1948), no.~3, 379--423.

\bibitem[SJ80]{shore1980axiomatic}
John~E Shore and Rodney~W Johnson, \emph{Axiomatic derivation of the principle of maximum entropy and the principle of minimum cross-entropy}, IEEE Transactions on Information Theory \textbf{26} (1980), no.~1, 26--37.

\bibitem[Sky87]{skyrms1987updating}
Brian Skyrms, \emph{Updating, supposing, and maxent}, Theory and Decision \textbf{22} (1987), no.~3, 225--246.

\bibitem[SSBD14]{shalev2014understanding}
Shai Shalev-Shwartz and Shai Ben-David, \emph{Understanding machine learning: from theory to algorithms}, Cambridge University Press, 2014.

\bibitem[Tou09]{touchette2009large}
Hugo Touchette, \emph{The large deviation approach to statistical mechanics}, Physics Reports \textbf{478} (2009), no.~1-3, 1--69.

\bibitem[Var16]{varadhan2016large}
S~R~Srinivasa Varadhan, \emph{Large deviations}, Courant Lecture Notes, vol.~27, American Mathematical Society, 2016.

\bibitem[Wil80]{williams1980bayesian}
Peter~M Williams, \emph{Bayesian conditionalisation and the principle of minimum information}, The British Journal for the Philosophy of Science \textbf{31} (1980), no.~2, 131--144.

\bibitem[Wol96]{wolpert1996reconciling}
David~H Wolpert, \emph{Reconciling {B}ayesian and non-{B}ayesian analysis}, Maximum Entropy and Bayesian Methods: Santa Barbara, California, USA, 1993, Springer, 1996, pp.~79--86.

\end{thebibliography}

\end{document}